\documentclass{amsart}

\usepackage{amssymb, amsmath, amsthm, hyperref}

\usepackage{graphicx}

\usepackage[english]{babel}

\author{Rafael Torres and Jonathan Yazinski}

\title[On the type and number of type change loci]{On the number of type change loci of a generalized complex structure}

\address{Mathematical Institute - University of Oxford\\ 24-29 St Giles\\Oxford\\OX1 3LB\\England}

\email{torres@maths.ox.ac.uk}

\address{McMaster University - Dept. Mathematics and Statistics\\ 1280 Main Street West\\L8S 4K1\\Hamilton, ON.}

\email{yazinski@math.mcmaster.ca}



\theoremstyle{plain}

\newtheorem{theorem}[equation]{Theorem}

\newtheorem{proposition}[equation]{Proposition}

\newtheorem{remark}[equation]{Remark}

\theoremstyle{definition}

\newtheorem{definition}[equation]{Definition}

\newcommand{\R}{\mathbb{R}}

\newcommand{\Z}{\mathbb{Z}}

\newcommand{\N}{\mathbb{N}}

\newcommand{\C}{\mathbb{C}}

\begin{document}

\maketitle

\emph{Abstract}: In this note, we describe a procedure to construct generalized complex structures with an arbitrarily large number of type change loci on products of the circle with a connected sum of closed 3-manifolds. The loci need not be isotopic.\\

\section{Introduction}

Generalized complex structures \cite{[H], [G]} are a simultaneous generalization of complex and symplectic structures that are obtained via the search for complex structures on Courant algebroids over a smooth manifold. Generalized complex geometry \cite{[H], [G]} provides an unifying language for several seemingly different features of symplectic and complex geometries. Important examples of such common themes are the generalized complex analogues for the $(p, q)$-decomposition of forms from complex geometry, and the analogues for symplectic and Lagrangian submanifolds.\\

A fundamental feature of these structures, is that the type of a generalized complex structure in dimension greater than two is not necessarily constant. It can jump along a codimension two submanifold called the \emph{type change locus}. In dimension four, the type jumps from a symplectic type to a complex one along a 2-torus, which inherits the structure of an elliptic curve. This was shown to be a generic phenomena in \cite{[G], [CG1], [CG2]}, where structures with at most one type change locus were built. This kind of generalized complex structures can be described as symplectic structures that acquire a singularity along a codimension two submanifold.\\

Recent constructions of inequivalent smooth structures on 4-manifolds with small topological invariants \cite{[Go], [FPS], [BK], [BK2]} were used in \cite{[T]} to produce examples of generalized complex structures with several type change loci on a myriad of 4-manifolds by applying the cut-and-paste operations \cite{[Lu], [ADK], [CG1], [CG2]}. Gualtieri and Cavalcanti asked the first author of this note the natural question: \emph{do generalized complex structures with arbitrarily many type change loci exist?} This question was answered in the affirmative in \cite{[T1]}  (see Section \ref{Section CC}) and independently in \cite{[GH]}.\\

In this note, we describe a procedure in Section \ref{Section 4} (cf. \cite{[FPS], [BK], [T]}) to construct generalized complex structures with as many type change loci as you want, as long as you want at least one, on product manifolds \begin{equation} S^1\times M^3,\end{equation} where $M^3$ is a closed 3-manifold formed as a connected sum of manifolds with a variety of choices of fundamental group. The following result exemplifies the procedure, as well as the choice of 3-manifold $M^3$. We denote the connected sum of $t$ copies of a manifold $M$ by $\# t\cdot M$, and a surface of genus $g$ by $\Sigma_g$; in particular, $\Sigma_0 = S^2$ and $\Sigma_1 = T^2$. The Lens spaces considered are $L(p_i, 1)$, and we assume $p_i\geq 2$ for $i\in \N$; in particular, $L(2, 1)$ is the real projective 3-space $\mathbb{RP}^3$.

\begin{theorem}{\label{Theorem 1}} Let $a\in \{0, 1\}$, and let $b, c$ be nonnegative integers. The almost-complex 4-manifold
\begin{equation} S^1\times (S^3\#(\#a\cdot S^1\times \Sigma_g)\#(\# b\cdot S^1\times S^2) \# (\# c\cdot L(p_i, 1)))\end{equation} 

admits a generalized complex structure $\mathcal{J}_n$ with $n\in \N$ type change loci.

\end{theorem}

The procedure described in Section \ref{Section 4} produces generalized complex structures on a much larger set of connected sums of 3-manifolds. The type change loci need not be homologically equivalent as discussed in Section \ref{Section 3}. There is an overlap between Theorem \ref{Theorem 1} and Proposition \ref{Proposition OV}, and \cite[Theorems 3.1 and 4.7]{[GH]}. 
It has been known for a couple of years \cite{[T1]} that, building on recent progress on 4-manifold topology \cite{[Go], [FPS], [BK]}, generalized complex structures with arbitrarily many type change loci can be constructed using only \cite{[Lu], [ADK], [CG1], [CG2]} by small tweaks to the proofs in \cite{[T]} (see Section \ref{Section CC} and Remark \ref{Remark 1}). These operations were generalized in \cite{[GH]} to construct generalized complex structures with arbitrarily many type change loci as well. Moreover, the main objective of this note as exemplified in Theorem \ref{Theorem 1} is to produce generalized complex structures on manifolds of the form $S^1\times M^3$, while the results \cite{[GH]} apply to broader generality (please see Remark \ref{Remark 1}).

\subsection{Acknowledgements} We thank Ron Stern for sharing his expertise, which resulted in Section \ref{Section 5.1} and motivated the writing of this note. We thank Ryushi Goto and Kenta Hayano for their interest in this paper and for their input, which helped us improved the manuscript. R. T. thanks Gil Cavalcanti, Ryushi Goto, Marco Gualtieri, Kenta Hayano, Nigel Hitchin, Paul Kirk, and Stefano Vidussi for useful conversations/e-mail exchanges. Support from the Simons Foundation is gratefully acknowledged.


\section{Background}

\subsection{Type change loci of generalized complex structures}{\label{Section 2}} A generalized complex structure on a manifold is prescribed by its canonical bundle, whose definition we now recall following \cite{[G]}. Throughout the paper, we assume that $M$ is a closed smooth manifold.

\begin{definition} (cf. \cite[Definition 3.7]{[G]}). The \emph{canonical bundle} of a generalized complex structure $\mathcal{J}$ on the sum $TM\oplus T^{\ast}M$ is the line sub-bundle $K\subset \wedge^{\bullet} T^{\ast}_{\C}M$ that annihilates its $+i$-eigenspace.

The \emph{type} of a generalized complex structure $\mathcal{J}$ \cite[Definition 3.5]{[G]}  is the upper semi-continuous function
\begin{center}
type$(\mathcal{J}) = \frac{1}{2} dim_{\R} T^{\ast}M \cap \mathcal{J} T^{\ast}M$,
\end{center}

with possible values $\{0, 1, \ldots, n\}$, where $n = \frac{1}{2} dim_{\R} M$. 

\end{definition}

The type of $\mathcal{J}$ can be read off from the differential form that generates its canonical bundle $K$ as follows. A generator $\varphi \in K_x$ for the canonical line bundle at the point $x\in M$ is of the form \begin{equation} \varphi = e^{B + i \omega} \wedge \Omega,
\end{equation}

where $\Omega = \theta_1 \wedge \cdots \wedge \theta_k$ for a basis $\{\theta_1, \cdots, \theta_k\}$ for $L \cap T^{\ast}_{\C}M$, where $L$ is the $+i$-eigenbundle, and $B, \omega$ are the real and imaginary components of a complex 2-form. The type of $\mathcal{J}$ is given by the least nonzero degree $(k)$ of the differential form $\varphi$ \cite[Section 3.1]{[G]}\begin{center}
type$(\mathcal{J}) = deg (\Omega)$.
\end{center}

Although the type need not be constant and it may jump along closed submanifolds of codimension two  \cite{[G], [CG1], [CG2], [T], [T1], [GH]}, its parity does remain fixed throughout the manifold \cite[Section 4.1]{[G]}.

\subsection{Torus surgeries and geometric structures.}{\label{Section 3}}

Let $T \subset M$ be a 2-torus of self intersection zero inside a $4$-manifold, and let $\nu(T) \cong T^2\times D^2$ be its tubular neighborhood. Let $\{\alpha, \beta\}$ be the generators of $\pi_1(T)$ and consider the meridian $\mu_T$ of $T$ inside the complement $M - \nu(T)$, and push offs $S^1_{\alpha}, S^1_{\beta}$ of the generators $\{\alpha, \beta\}$ in $\partial \nu(T) \approx T^3$. There is no ambiguity regarding the choice of push offs, since in our constructions the manifold $M$ will be symplectic, the torus $T$ will be Lagrangian (or symplectic), and the push offs are taken with respect to the Lagrangian framing. The loops $S^1_{\alpha}$ and $S^1_{\beta}$ are homologous in $\nu(T)$ to $\alpha$ and $\beta$ respectively.\\

The manifold obtained from $M$ by performing a $(p, q, r)$-torus surgery on $T$ along the curve $\gamma:= p S^1_{\alpha} q S^1_{\beta}$ is defined as\begin{equation}M_{T, \gamma}(p, q, r) := (M - \nu(T)) \cup_{\varphi} (T^2 \times D^2),
\end{equation}

where the diffeomorphism $\varphi: T^2\times \partial D^2 \rightarrow \partial (M - \nu(T))$ used to glue the pieces together satisfies \begin{equation} \varphi_*([\partial D^2]) = p[S^1_{\alpha}] + q[S^1_{\beta}] + r[\mu_T]\end{equation} in $H_1(\partial (M - \nu(T))); \Z)$.\\

If $T$ is a Lagrangian torus, the manifold $M_{T, \gamma}(p, q, 1)$ admits a symplectic structure \cite{[Lu], [ADK]}. If the torus $T$ is a symplectic submanifold, then $M_{T, \gamma}(p, q, 0)$ admits a generalized complex structure with one type change locus \cite{[CG1]}. The type of the generalized complex structure $\mathcal{J}$ jumps from 0 to 2 along the core torus \begin{equation}\hat{T}:= T^2\times \{0\} \subset T^2\times D^2\end{equation} of each surgery. Such a torus is nullhomologous in $M_{T, \gamma}(p, q, 0)$.\\

Regarding the homology class of  the torus that acts as the type change locus we have the following. Assume the generalized complex structure has at least two type change loci $\hat{T}_i\subset X_T(p, q, 0)$ for $i = 1, 2$. Blowing up a point along a point on the submanifold $\hat{T}_1$, results in a generalized complex structure on $X_T(p, q, 0)\# \overline{\mathbb{CP}^2}$. Using the exceptional sphere introduced during the blow up, one sees that the homology class of the torus $\bar{T}_1= \hat{T}_1 \# \overline{\mathbb{CP}^1}$ is non-trivial, and it has self-intersection $-1$. The torus $\hat{T}_2$ is nullhomologous. In particular, the type change loci of a generalized complex structure need not represent the same homology class. 

\begin{remark}{\label{Observation 1}} Computations of fundamental groups of cut-and-paste constructions along submanifolds of codimension two are as challenging as they are essential when determining the topological type of the resulting manifold \cite{[BK]}. A $(p, q, 0)$-torus surgery carves out the torus $T$ and glues it back so that its meridian bounds a disk in $M_{T, \gamma}(p, q, 0)$. In particular, a generator of the group $\pi_1(M)$ is killed directly \cite{[T]}. Performing a $(p, q, 0)$-torus surgery allows us to skip these computations, and pin down directly the diffeomorphism type of $M_T(p, q, 0)$ as we do in the next section. Moreover, due to our goal of producing arbitrarily many type change loci, we will be generous with the number of these surgeries performed.
\end{remark}

\section{Surgical procedure and topological considerations.}{\label{Section 4}}

\subsection{Description of the construction and proof of Theorem \ref{Theorem 1}}{\label{Section 4.1}} The following procedure is adapted from \cite{[FPS], [BK]} to our purposes. Consider the manifold $T^2\times \Sigma_h$, and equip it with the product symplectic form $\pi_1^{\ast}\omega_{T^2}\oplus \pi_2^{\ast}\omega_{\Sigma_h}$, where $\pi_1$ and $\pi_2$ are the projections to the first and second factors respectively.  The tori $T_i$ involved in the surgeries are of the form \begin{equation}S^1\times S^1 \times \{pt\} \subset (S^1\times S^1)\times \Sigma_h = T^2\times \Sigma_h\end{equation}\begin{equation}{\label{Tori}}
\{pt\}\times S^1\times S^1 \subset S^1\times (S^1\times \Sigma_h) = T^2\times \Sigma_h,
\end{equation}

each of the corresponding surgery curves $\gamma_i$ carries a generator \begin{equation} y, a_1, b_1, \cdots, a_h, b_h\end{equation} 

of the group $\pi_1(T^2)\times \pi_1(\Sigma_h) \cong \pi_1(T^2\times \Sigma_h)$, where $\{x, y\}$ generate $\pi_1(T^2)$ and $\{a_i, b_i\}$ generate $\pi_1(\Sigma_g)$ as indicated in Figure 1. The loop $x$ will not be used as a surgery curve in any of our cut-and-paste constructions, and it generates the circle product factor in the resulting manifold $S^1\times M^3$.  In what follows we will abuse notation and denote by $x, y, a_i, b_i$ the loops that carry the generators of the group. The submanifolds $T_i$ are homologically essential Lagrangian tori, and the symplectic form can be perturbed \cite[Lemma 1.6]{[Go]} so that any of these tori becomes symplectic when needed. \\

\begin{figure}{\label{Figure 1}}
\begin{center}
\includegraphics[scale=0.20, viewport= 1550 80 450 600]{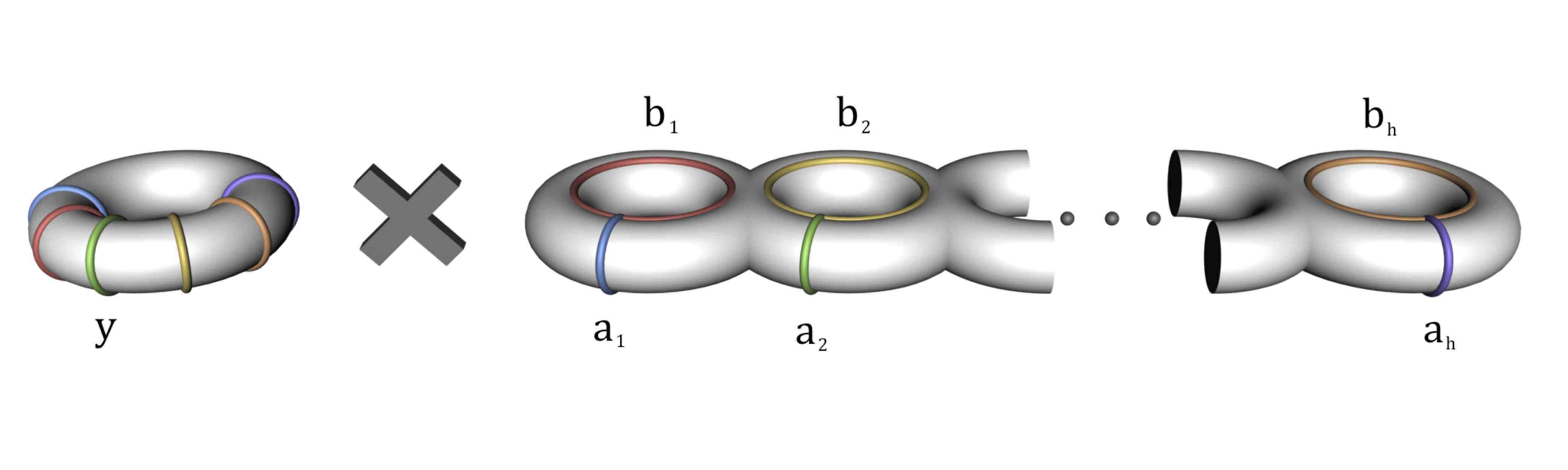}
\caption{Surgery curves and tori in $T^2\times \Sigma_h$}
\end{center}
\end{figure}

To fix ideas, and for the sake of clarity, we describe the construction of a generalized complex structure with an arbitrary number of type change loci on the manifolds \begin{equation}S^1 \times (S^3\#t\cdot S^1\times S^2)\end{equation} and \begin{equation}T^2\times \Sigma_g.\end{equation} We then explain the modifications on the argument required to construct all the manifolds of Theorem \ref{Theorem 1}. The choice of initial manifold to undergo surgery is $T^2\times \Sigma_h$, where the genus  of the surface is a function $h = h(a, b, c, n)$, where $a, b, c, n$ are as in Theorem \ref{Theorem 1}, that depends on the desired 3-manifold $M^3$, and on the number of type change loci one wants for the generalized complex structure to have. \\

$\bullet$ Let us construct a generalized complex structure $\mathcal{J}_n$ with $n\in \N$ type change loci on $T^2\times \Sigma_g$. Set $h = n + g$, and consider $T^2\times \Sigma_h$ as our starting manifold equipped with the symplectic structure $\pi^{\ast}\omega_{T^2} \oplus \pi^{\ast}\omega_{\Sigma_h}$. The group $\pi(T^2\times \Sigma_h)$ is generated by the loops \begin{equation} x, y, a_1, b_1, \cdots, a_g, b_g, a_{g + 1}, b_{g + 1}, \cdots, a_{g + n}, b_{g + n}.\end{equation}

We perform a total of 2n $(p, q, r)$-torus surgeries along tori $\{T_i\}$, each surgery curve $\gamma_i$ on the 2-torus $T_i$ given by one of the loops $a_{g + 1}, b_{g + 1}, \cdots, a_{g + n}, b_{g + n}$, to obtain a manifold diffeomorphic to $T^2\times \Sigma_g$. Setting either $(p, q) = (1, 0)$ or $(p, q) = (0, 1)$ according to the choice of $\gamma_i$, we perform $n$ surgeries with coefficient $r = 1$ \cite{[Lu], [ADK]}, and $n$ surgeries with coefficient $r = 0$ \cite{[CG1]}. Denote by $X(g, n)$ the resulting manifold, which admits a generalized complex manifold with $n$ type change loci by \cite{[CG1]}. We will prove in the next section that $X(g, n)$ is diffeomorphic to $T^2\times \Sigma_g$.\\

$\bullet$ We construct a generalized complex structure $\mathcal{J}_n$ with $n\in \N$ type change loci on $S^1 \times (S^3\#t\cdot S^1\times S^2)$. Set $h = t + n$. The group $\pi(T^2\times \Sigma_h)$ is generated by the loops \begin{equation} x, y, a_1, b_1, \cdots, a_t, b_t, a_{t + 1}, b_{t + 1}, \cdots, a_{t + n}, b_{t + n}\end{equation}
(see Figure 1). We choose the tori \ref{Tori} that have on them the surgery curves $\gamma_i = a_i$ for $i =1, \ldots, t + n$ and $\gamma_j = b_j$ for $j = t, \cdots, t + n$, and apply $(1, 0, r)$- and $(0, 1, 1)$-torus surgeries on them, respectively. Since we want a generalized complex structure $\mathcal{J}_n$ with $n$ type change loci, we will perform along the surgery curves $\gamma_i$ $n$ torus surgeries with coefficient $r = 0$, and $t$ torus surgeries with coefficient $r = 1$. First, we apply all the $(0, 1, 1)$- and $(1, 0, 1)$-torus surgeries, to obtain a symplectic manifold \cite{[ADK], [Lu]}. We then perturb the symplectic structure so that the unused tori become symplectic, and then apply the remaining surgeries with $r = 0$ \cite{[CG1]}. If $t = 0$, then we do a $(0, 1, 1)$-torus surgery along the curve $\gamma = y$. Denote by $M(t, n)$ the manifold obtained after the surgeries. We will prove in the following section that $M(t, n)$ is diffeomorphic to $S^1 \times (S^3\#t\cdot S^1\times S^2)$. \\

$\bullet$ The only remaining unaccounted diffeomorphism type in Theorem \ref{Theorem 1} corresponds to connected sums that include copies of Lens spaces $L(p_i, 1)$, and we proceed to explain their construction. Recall that the genus of the surface of the initial manifold is determined by $h = h(a, b, c, n)$, where $b$ is the number of copies of $S^1\times S^2$ and $c$ is the number of copies of Lens spaces. In the previous construction of $S^1\times (S^3\#t\cdot S^1\times S^2)$ some of the surgery loops $\gamma_i = b_i$ were not used during the surgeries, and they give rise to the $S^1\times S^2$ factors. To obtain the desired manifold, one can set $b\mapsto b + c$, and perform a $(p, 0, 1)$- or $(0, p, 1)$-surgery on a torus along the surgery curve with $p\geq 2$ \cite{[Lu], [ADK]}. As mentioned in Remark \ref{Observation 1} (cf. \cite{[T]}), after performing at least one surgery with $r=0$, an $(0, p, 1)$-surgery along the surgery loop $\gamma = b$ introduces the relation $b^p = 1$ to the fundamental group. In particular, such a surgery results in the final manifold containing a Lens space instead of a copy of $S^1\times S^2$ as we will see in the following section.

\begin{figure}{\label{Figure 3}}
\begin{center}
\includegraphics[scale=0.60, viewport= 110 50 460 700]{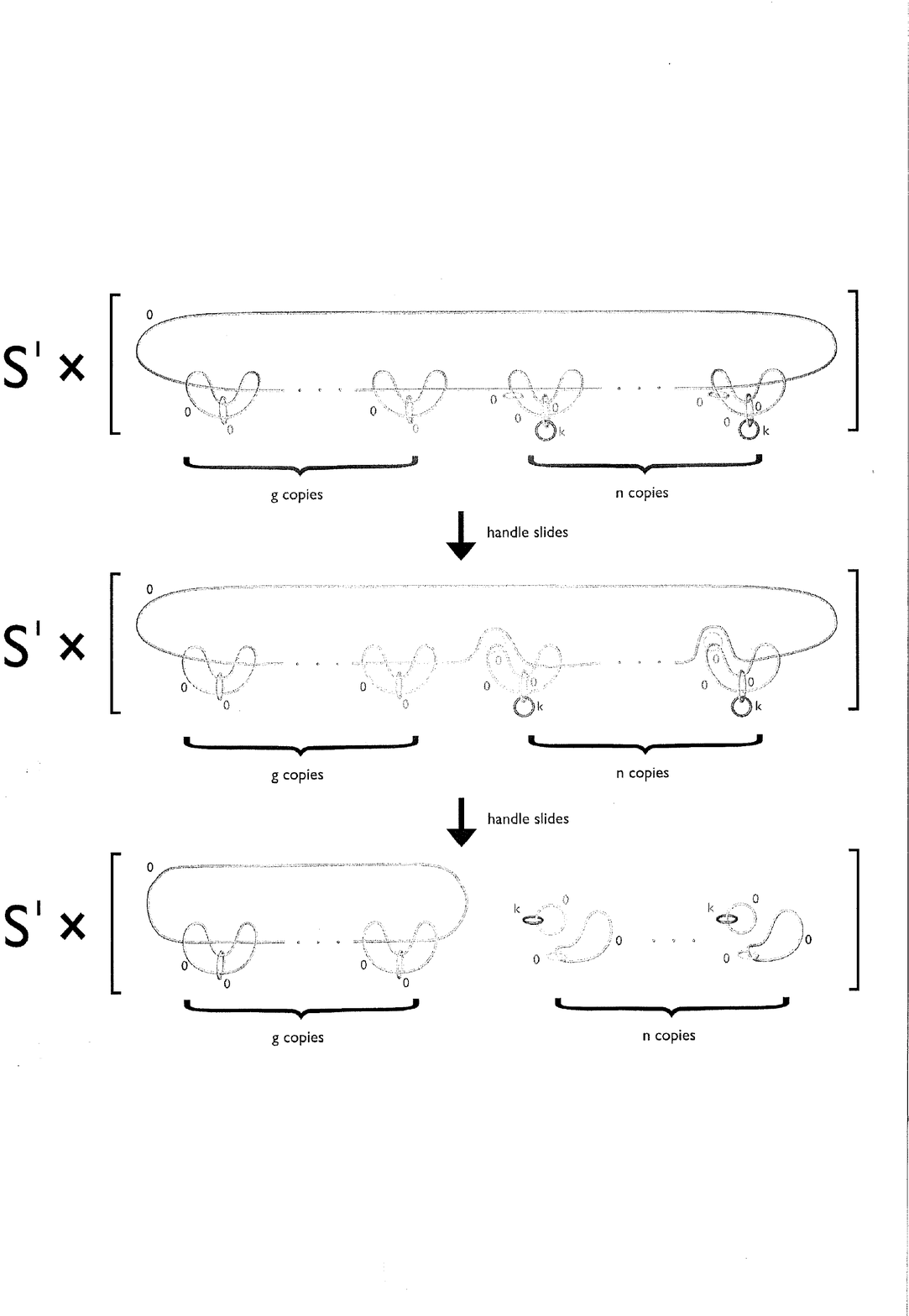}
\caption{Diffeomorphism with $S^1\times (S^1\times \Sigma_g)$}
\end{center}
\end{figure}

\begin{figure}{\label{Figure 2}}
\begin{center}
\includegraphics[scale=0.60, viewport= 110 50 460 700]{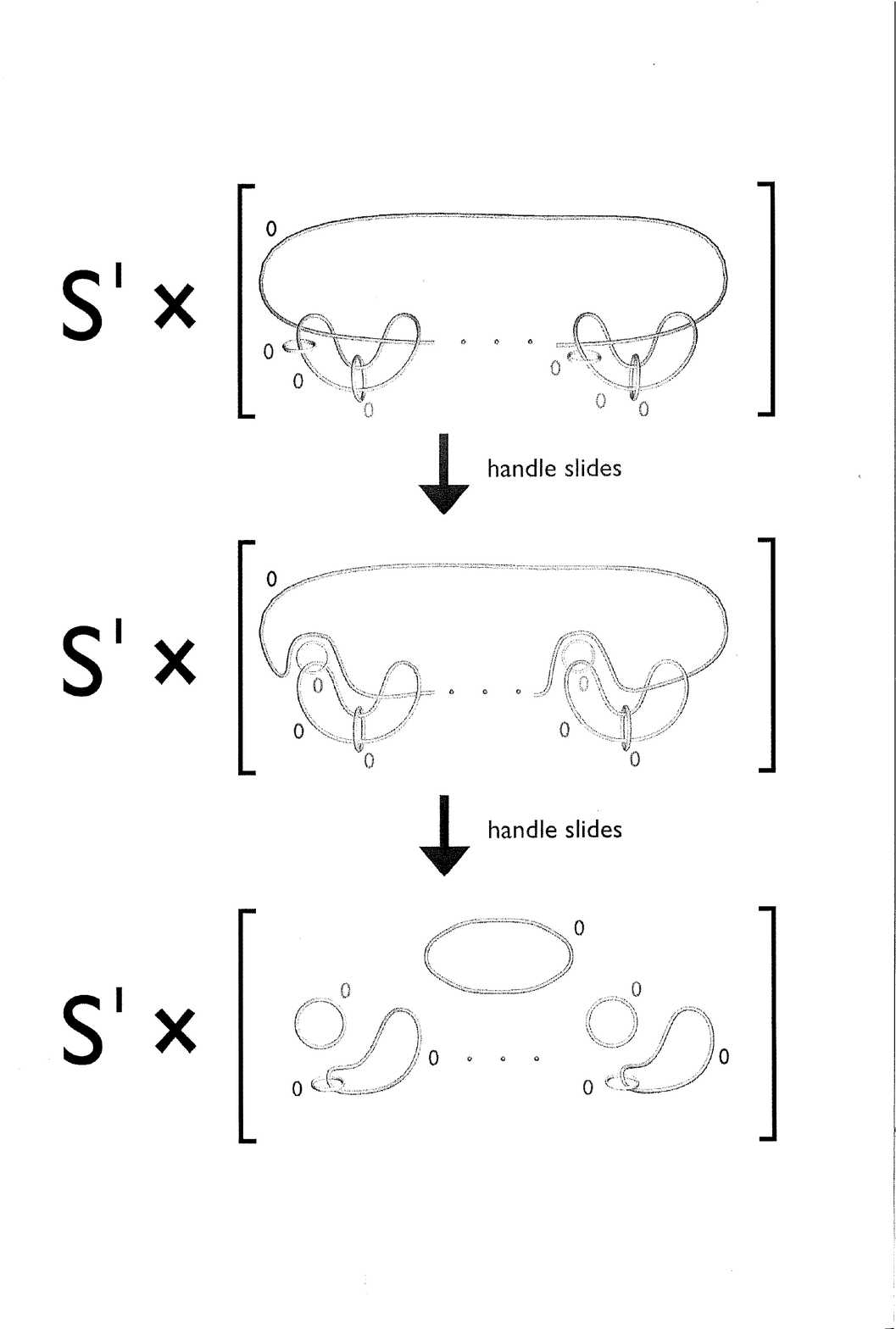}
\caption{Diffeomorphism with $S^1\times (\#t\cdot S^1\times S^2)$}
\end{center}
\end{figure}

\subsection{Diffeomorphisms.}{\label{Section 5.1}} In order to pin down the diffeomorphism types of the manifolds that were constructed in Section \ref{Section 3}, we use the property that a diffeomorphism $\phi: M_1^3\rightarrow M_2^3$ of 3-manifolds extends to a diffeomorphism of the product 4-manifolds $\hat{\phi}: S^1\times M_1^3 \rightarrow S^1\times M_2^3$. We are indebted to Ron Stern, from whom we learned the following argument. We think of the surgical tori of Section \ref{Section 4} as the product of a circle and a thick 2-torus
\begin{equation}
S^1\times (S^1\times D^2) \subset S^1\times (S^1\times \Sigma_h).
\end{equation}

The torus surgeries described in Section \ref{Section 3} are thought of as \begin{center}$S^1 \times$ Dehn surgery,\end{center}

and we restrict our attention to the effect that the Dehn surgeries have on the 3-dimensional factor $S^1\times \Sigma_h$ to conclude on the existence of the claimed diffeomorphism using three-dimensional link calculus \cite{[R], [K], [FRr]} (see \cite[Chapter VI]{[PS]} for a concise exposition). 

\begin{proposition}{\label{Proposition D}} Let $X(n)$ and $Y(n)$ be the manifolds constructed in Section \ref{Section 4}. There exist diffeomorphisms\begin{equation}{\label{Item 2}} \Psi': X(g, n)\longrightarrow S^1\times (S^1\times \Sigma_g)\end{equation}\begin{equation}{\label{Item 1}}
\Psi: M(t, n)\longrightarrow S^1\times (S^3\#(\#t\cdot S^1\times S^2))
\end{equation}
\end{proposition}

\begin{proof} Item \ref{Item 2}: The existence of the claimed diffeomorphism follows from the three dimensional Kirby calculus argument shown in Figure 2.  We begin with the standard diagram of $S^1\times \Sigma_{g + n}$ given by the top part of Figure 2 removing the small horizontal yellow 0-framed handles linked to the red 0-framed handles, and the lower vertical purple k-framed handles linked to the green 0-framed handles.  The manifold $S^1\times \Sigma_{g + n}$ is obtained by surgeries on the components of the link in $S^3$. The larger 0-framed horizontal handle represents the loop $y$, while the handles linked to it represent the loops $\{a_i\}$. The vertical 0-framed handles linked to the handles representing the $a_i$ loops represent the loops $\{b_i\}$. The first $g$ linked 0-framed handles arising from the loops $\{y, a_1, b_1, \cdots, a_g, b_g\}$ go untouched during the surgical procees, and will be the final $S^1\times \Sigma_g$ block \cite{[R]}. The surgeries on the loops $\{a_j, b_j: j = g + 1, \ldots, g + n\}$ are represented by attaching the aforementioned linked horizontal and vertical, respectively, as in the right part of the top diagram in Figure 2. After the indicated handleslides, we arrive to the bottom part of Figure 2, where each manifold obtain by surgery on a pair of linked handles is a 3-sphere. The claim now follows. 

Item \ref{Item 1}: We begin with the standard diagram of $S^1\times \Sigma_h$ given by the top part of Figure 3 with the small horizontal 0-framed handles, and the lower vertical 1-framed handles removed from the picture.  The larger 0-framed horizontal handle represents the loop $y$, while the handles linked to it represent the loops $\{a_i\}$. The vertical 0-framed handles linked to the handles representing the $a_i$ loops represent the loops $\{b_i\}$. The loops $\{y, b_i\}$ will provide the $S^1\times S^2$ factors on the connected sum: notice that the handlebody of $S^1\times S^2$ is an uknotted 0-framed circle \cite[Proposition 14.4]{[PS]}. Performing a $(1, 0, r)$-torus surgery along the surgery curve $\gamma_i = a_i$ is represented by attaching a vertical 0-framed handle to the diagram such that it links the surgery curve. The $(0, 1, 1)$-torus surgeries are indicated in the diagram as 1-framed handles linked to the vertical 0-framed handles that represent the loops $\{b_j\}$.

Perform the handle slides indicated in the diagram in the middle of Figure 3. In the bottom diagram of Figure 3, the linked circles with 0- or 1-framing represent  the connected sum of copies of the 3-sphere, while the $t$ unlinked 0-framed handles  represents $(\#t\cdot S^1\times S^2)$ \cite[Proposition 14.4]{[PS]}. This implies that the 3-manifold is diffeomorphic to $(\#t\cdot S^1\times S^2)$, and therefore $M(t, n) = S^1\times (\#t\cdot S^1\times S^2)$. Performing a $(0, 1, r)$-torus surgery along $\gamma = y$ amounts to adding a linked handle to the bigger horizontal 0-framed handle, and the same handle slides prove that the resulting manifold is diffeomorphic to $S^1\times S^3$.
\end{proof}

Finally, to conclude on the diffeomorphism type of the remaining cases of Theorem \ref{Theorem 1}, one argue as follows. The only remaining case of the connected sum of 3-manifolds are those that include Lens spaces. For this purpose we consider $L(p, 1)$ as obtained from performing surgery on the unknot in $S^3$ with framing $p$ or by surgery on a Hopf link with framings $p + 1$ and $1$ \cite[Chapter 9]{[FRr]}, \cite[Chapter VI]{[PS]}. The surgery described in the last paragraph of Section \ref{Section 4.1} changes the corresponding 0-framed unknots or Hopf links with framings 0 and k in Figure 1 or Figure 2 to a $p$-framed unknots or the Hopf links with the corresponding framings, respectively. A weaker statement on the existence of Lens spaces factors in the connectd sums of Theorem \ref{Theorem 1} is obtained by computing the fundamental group of the resulting manifold as done in \cite{[T]}, albeit this is not sufficient to conclude on the homeomorphism type of the Lens space. Indeed, recall that two Lens spaces $L(p, q_1)$ and $L(p, q_2)$ have cyclic fundamental group of order p, yet they are homeomorphic only if $q_1 = q_2$ mod $p$ \cite[Section 11.3]{[PS]} (see Remark \ref{Remark 1}). 

\subsection{Produce of coupling two areas of research}{\label{Section CC}}  The following proposition is two-fold, and it is a direct consequence of the recent outstanding progress on 4-manifold theory \cite{[Go], [FS2], [V1], [FPS], [BK], [BK2]}. First, we point out that existence of a myriad of examples of generalized complex structures with arbitrarily many type change loci is an immediate corollary of coupling constructions of generalized complex structures \cite{[CG1], [CG2]} with the study of existence of nonisotopic symplectic submanifolds $T_j \subset M$ representing the same homology class \cite{[FS2], [V1]}. Second, as mentioned in the introduction, we exemplify how such structures are constructed using a small tweak to the proofs in \cite{[T]}. Hence the overlap in the claims of the proposition.

\begin{proposition}{\label{Proposition OV}}(cf. \cite{[CG2], [T1], [GH]}). There exists a generalized complex structure $\mathcal{J}_k$ on a closed 4-manifold with $k$ type change loci. 

Moreover, such closed 4-manifold can be taken to be an elliptic surface $E(n)$ or every connected sum of copies of $\mathbb{CP}^2$ and $\overline{\mathbb{CP}^2}$ that admits an almost-complex structure.

\end{proposition}

\begin{proof} Let $E(n)$ be an elliptic surface, and let $F$ be a fiber. According to \cite[Theorem 5.2]{[FS2]}, \cite[Theorem 1.1]{[V1]} there exists an infinite set of pairwise nonisotopic symplectic 2-tori $\{T_j: j\in \N\}$, which are embedded in $E(n)$ that represent the homology class $2d[F]$. These are disjoint tori of self-intersection zero. The proposition follows from \cite{[CG1]} by simultaneously applying $(1, 0, 0)$-torus surgery on $k$ of these tori for $k\in \N$.

The proof the second claim follows from coupling the results in \cite{[CG2]} with tweaks of \cite[Section 4]{[T]} as we know explain. Consider two fibers $F_1, F_2 \subset E(n)$, and the product of a torus and a surface of genus $h$ equipped with the product symplectic form $(T^2\times \Sigma_h, \pi^{\ast}\omega_{T^2} \oplus \pi^{\ast}\omega_{\Sigma_h})$. Denote by $T:= T^2\times \{pt\} \subset T^2\times \Sigma_h$. This is a symplectic submanifold of self-intersection zero. Construct the fiber sum along $F_2$ and $T$\begin{equation}
\widetilde{X(n, h)}:= E(n) \#_{F_2 = T} T^2\times \Sigma_h,
\end{equation}

which admits a symplectic structure \cite[Theorem 1.3]{[Go]}. The symplectic manifold $\widetilde{X(n, h)}$ contains enough homologically essential Lagrangian tori such that each one carries a surgery loops that is a generator of the fundamental group $\pi_1(\widetilde{X(n, h)}) \cong \pi_1(\Sigma_h)$. A rigorous description of the surgery tori can be found in Lemma 31 of the arXiv version of \cite{[T]}: these computations are nothing more than an extension of the computations done in \cite{[BK]} and \cite[Section 3.3]{[BK2]}. 

As in the previous section, the genus of $\Sigma_h$ is a function on the number of type change loci one wants to construct, and one applies simulstaneously $k$ $(1, 0, 0)$-torus surgeries \cite{[CG1]} after applying \cite[Lemma 1.6]{[Go]}, and the remaining $(1, 0, 1)$-torus surgeries \cite{[Lu], [ADK]}. Let $X(n, h)$ be the simply connected manifold obtained after applying all the torus surgeries. It admits a generalized complex structure with $k$ type change loci by \cite{[CG1]}. The existence of a diffeomorphism between $X(n, h)$ and $E(n)$ can be argued as follows. Deconstruct the construction into a) the torus surgeries applied to the block $T^2\times \Sigma_h$ and b) the fiber sum used to construct $\widetilde{X(n, h)}$. The argument used to prove Proposition \ref{Proposition D} shows that the surgeries transform $T^2\times \Sigma_h$ into $T^2\times S^2$, and the fiber sum $E(n)$ and $T^2\times S^2$ along a fiber and $T^2\times \{pt\}\subset T^2\times S^2$ preserves the diffeomorphism class \cite[Remark 1.3]{[U]}. Thus, $X(n, h)$ is diffeomorphic to $E(n)$.
The claim regarding the almost-complex connected sums $m_1\mathbb{CP}^2 \# m_2 \overline{\mathbb{CP}^2}$ ($m_1$ odd and $m_2$ a nonnegative integer number) now follows verbatim as \cite[Example 5.3]{[CG2]} by applying a $(1, 0, 0)$-torus surgery to the unused fiber $F_1\subset E(n)$. It is proven in \cite[Section 3]{[G2]} and \cite[Appendix]{[CG2]} that this last surgery decomposes $E(n)$ into $(2n - 1)\mathbb{CP}^2\# (10n - 1)\overline{\mathbb{CP}^2}$. The claim follows by blowing up/blowing down branes that intersect the complex loci exactly as it is done in \cite[Section 3]{[CG2]}. 
\end{proof}

The previous argument gives a method to produce generalized complex structures with arbitrarily many type change loci on any symplectic 4-manifold that contains an embedded 2-torus, which is either symplectic or  homologically essential Lagrangian torus. The precise statement is as follows.

\begin{proposition} Every symplectic 4-manifold that contains a homologically essential 2-torus of self-intersection zero admits a generalized complex structure $\mathcal{J}_n$ with $n\in \N$ type change loci.

\end{proposition}

\begin{remark}{\label{Remark 1}} The reader will benefit from comparing these constructions with those in \cite[Theorems 3.1 and 4.7]{[GH]}. The results \cite[Theorem 3.1, Theorem 4.7]{[GH]} of Goto and Hayano can be invoked to prove Theorem \ref{Theorem 1} and Proposition \ref{Proposition OV}. For example, items in Theorem \ref{Theorem 1} were constructed using multiplicity 0 \cite{[CG1]} and 1 \cite{[Lu], [ADK]} log transforms, and can also be constructed using log transforms of more general multiplicity as in \cite{[GH]}. As it was informed by the first author to Goto and Hayano after learning about their preprint, their results on their first posting on the arXiv can also be used to construct generalized complex structures with arbitrarily many type change loci on a myriad of manifolds by building on \cite{[CG2]}, including the almost-complex connected sums $m_1\mathbb{CP}^2\# m_2\overline{\mathbb{CP}^2}$, as well as every manifold considered in \cite{[CG1], [CG2], [T]}.

There is a clear advantage of \cite{[GH]} over the tools considered in this preprint, and we point out two instances to exemplify it. Modulo the existence of a generalized complex structure, the procedure of Section \ref{Section 3} produces many more 3-manifolds as in the statement of Theorem \ref{Theorem 1} by applying more general $(p, q, r)$-torus surgeries, including copies of the Lens spaces $L(p, q)$ in the connected sum. Using \cite{[GH]} one can conclude the existence of a generalized complex structure with arbitrarily many type change loci in these manifolds as well. Second, Goto and Hayano's technology equips the almost-complex connected sums $m_2(S^2\times S^2)$ with generalized complex structures with arbitrarily many type change loci, thus extending \cite[Theorem 9]{[T]}. 
\end{remark}

\end{document}